\documentclass[11pt]{amsart}

\usepackage[T1]{fontenc}
\usepackage{lmodern}
\usepackage{graphicx}
\usepackage{mathrsfs}
\usepackage{color}
\usepackage{amssymb}
\usepackage{amsmath}
\usepackage{mathtools}
\usepackage[margin=1.4in]{geometry}
\usepackage[hyperindex,breaklinks,colorlinks=true,
  linkcolor=blue,citecolor=blue,urlcolor=blue]{hyperref}
\usepackage{microtype}
\newtheorem{theorem}{Theorem}
\newtheorem{proposition}[theorem]{Proposition}
\newtheorem{lemma}[theorem]{Lemma}
\newtheorem{corollary}[theorem]{Corollary}
\theoremstyle{definition}

\theoremstyle{definition}
\newtheorem{remark}[theorem]{Remark}

\newcommand{\R}{\mathbb R}

\newcommand{\Vol}{\operatorname{Vol}}
\newcommand{\diam}{\operatorname{diam}}
\newcommand{\tr}{\operatorname{tr}}
\newcommand{\Ric}{\operatorname{Ric}}
\newcommand{\supp}{\operatorname{supp}}
\newcommand{\Div}{\operatorname{div}}

\newcommand{\bangle}[1]{\left\langle #1 \right\rangle}
\title[Stable Capillary Minimal Hypersurfaces in $\mathbb{R}^4_+$]
{Stable Capillary Minimal Hypersurfaces in $\R^4_+$}
\author{}
\author{Jia Li}
 \address{(J.L.)School of Mathematical Sciences\\
 		Xiamen University\\
 		361005, Xiamen, P.R. China}
 \email{lijiamath@stu.xmu.edu.cn}
 \author{Chenghang Lu}
\address{(C.L.)Albert-Ludwigs-Universit\"at Freiburg,
Mathematisches Institut, Ernst-Zermelo-Str.~1,
D-79104 Freiburg, Germany}
\email{chenghang.lu@math.uni-freiburg.de}
\thanks{C.L. is supported by  DFG,  through
IRTG 3132 (Grant No.~545046569).}

\author{Chao Xia}
\address{(C.X.) School of Mathematical Sciences,
Xiamen University, 361005, Xiamen, P.~R.~China}
\email{chaoxia@xmu.edu.cn}
\thanks{C.X. is supported by NSFC (Grant No. 12271449, 12671069, 12526203, 12526102) and the Natural Science Foundation of Fujian Province of China (Grant No. 2024J011008).}

\begin{document}

\begin{abstract}
We prove that a complete, two-sided, stable capillary minimal hypersurface in
$\R^4_+$ is flat provided its contact angle $\theta$ satisfies
$|\cot\theta|<\kappa_*$, where
$\kappa_*=4\Gamma(3/4)^2/\Gamma(1/4)^2\approx 0.45$.
\end{abstract}
\maketitle

\section{Introduction}

The classical Bernstein problem asks whether an entire minimal graph in
$\R^{n+1}$ must be an affine hyperplane.  The answer is affirmative for
$n\leq7$ and negative for $n\geq8$; the relevant works include those of
Bernstein, Fleming, De Giorgi, Almgren, Simons, and
Bombieri--De Giorgi--Giusti
\cite{Bernstein1927,Fleming1962,DeGiorgi1965,Almgren1966,Simons,BDGG}.
More generally, the stable Bernstein problem asks whether every complete
two-sided stable minimal hypersurface in $\R^{n+1}$ is flat. It was resolved
for $n=2$ independently by
do Carmo--Peng, Fischer-Colbrie--Schoen, and Pogorelov
\cite{doCarmoPeng,FischerColbrieSchoen,Pogorelov}. Under the intrinsic Euclidean
volume growth condition, Schoen--Simon--Yau \cite{SchoenSimonYau} proved this rigidity result for
$n\leq5$. With extrinsic Euclidean volume growth, Schoen--Simon
\cite{SchoenSimon} proved this result for $n\leq6$ for embeddings,
and Bellettini \cite{Bel23} for immersions. Later, it has been proved that the intrinsic and extrinsic Euclidean volume growth conditions are
equivalent \cite{FloritSimon} for stable minimal hypersurfaces. More recently, Chodosh--Li
\cite{ChodoshLiBernstein} proved that every complete two-sided stable minimal
hypersurface in $\R^4$ is flat, without any volume growth assumption.
Alternative proofs were subsequently given by Chodosh--Li \cite{ChodoshLi},
using warped $\mu$-bubbles, and by Catino--Mastrolia--Roncoroni
\cite{Catino-Mastrolia-Roncoroni}, using conformal Bakry--\'{E}mery geometry.
More recently, Cabr\'{e}--Catino--Mari--Mastrolia--Roncoroni
\cite{CabreEtAl} gave a proof based on sharp gradient estimates for the
Green kernel under spectral Ricci bounds, and Hong--Wang
\cite{HongWangRadial} gave a short proof using mixed radial volume comparison.
The volume growth method of \cite{ChodoshLi} was also developed further to
prove stable Bernstein theorems in $\R^5$ and $\R^6$ \cite{CLMS,Mazet}; see
also \cite{LiXia}. The $n=6$ case has been claimed by Hong-Li-Wang \cite{HLW} using a Green function method.

The corresponding boundary case is that of capillary minimal hypersurfaces in \(\mathbb{R}^{n+1}_{+}\). A capillary minimal hypersurface is the critical point of the energy functional \(E(t)\) with respect to all compactly supported admissible variations, where \(E(t)\) is defined by
\[
E(t)=A(t)-\cos\theta\,W(t),
\]
where \(A(t)\) denotes the area and \(W(t)\) denotes the wetting area on the supporting hyperplane \(\partial\mathbb{R}^{n+1}_{+}\). A capillary minimal hypersurface is called stable if
the second variation of $E(t)$ is nonnegative. 

Bernstein theorem for capillary minimal graphs over $\R^{n}_+$ have been
 established in $n=2, 3, 4$, by following a more general result of Hong--Saturnino \cite{HongSaturnino} in $n=2$ (see also \cite{LiZhouZhu,MP21}), a classification of capillary minimal cones by Edelen--Li--Zhu \cite{EdelenLiZhu} and Wang--Zhang \cite{WangZhang} in $n=3,4$.
For $5\le n\le 7$, the corresponding statement holds under additional restrictions on the contact angle $\theta$; see \cite{LiZhouZhu,EdelenLiZhu,WangZhang}. It remains an open and challenging problem to determine whether these restrictions can be removed.

Stable Bernstein theorem for capillary minimal hypersurfaces in $\R^{3}_+$ has been established by
Hong--Saturnino \cite{HongSaturnino} by using the  method of Fischer--Colbrie and Schoen \cite{FischerColbrieSchoen}.

Li-Zhou-Zhu \cite{LiZhouZhu} proved stable capillary Bernstein theorem in $\R^{n+1}_+$ for $3\leq n\leq 5$, provided intrinsic Euclidean volume growth condition and restricted angle range, following the method of Schoen-Simon-Yau \cite{SchoenSimonYau}. Recently,  Wang-Zhang \cite{WangZhang} developed Schoen-Simon regularity and compactness theory for stable capillary minimal
hypersurfaces and apply it to stable capillary Bernstein theorem for any contact angle in $\R^{n+1}_+$ for $3\leq n\leq 4$ under embededness and extrinsic Euclidean volume growth condition. 

In this paper, we apply the $\mu$-bubble method of Chodosh-Li \cite{ChodoshLi} to prove the intrinsic Euclidean volume growth and establish stable capillary Bernstein theorem in $\R^{4}_{+}$, provided that the contact angle $\theta$ belongs to a certain range.
Denote
\begin{equation*}
 \kappa_*=4\frac{\Gamma(3/4)^2}{\Gamma(1/4)^2}
 =0.4569465810\ldots,
 \qquad
 c_*:=\frac{\kappa_*}{\sqrt{1+\kappa_*^2}}
 =0.4156120919\ldots .
\end{equation*}

We assume that
\begin{equation}
 |\cos\theta|<c_*, \hbox{or equivalently }
|\cot\theta|<\kappa_*.
 \label{eq:main-angle}
\end{equation}

The corresponding interval for $\theta$  is
$$
65.442131^\circ < \theta < 114.557869^\circ.
$$
 
The main theorem of this paper is as follows.

\begin{theorem}
\label{thm:main}
Any complete, connected, two-sided stable capillary minimal immersion
\[
 X:(M^3,\partial M)\longrightarrow
 (\R^4_+,\partial\R^4_+)
\]
with contact angle $\theta$ satisfying \eqref{eq:main-angle} is flat.
\end{theorem}

Combining Theorem~\ref{thm:main} with the standard point-picking and blow-up argument used in \cite[Theorems 1.2 and 1.3]{ChodoshLiBernstein} and with its capillary boundary version in \cite[Theorem 1.6]{HongSaturnino}, we obtain the following curvature estimate.

\begin{corollary}
\label{cor:local-curvature-estimate}
Let $(N^4, g)$ be a complete Riemannian manifold with smooth boundary. Assume that $N$ has sectional curvature
bounded above by $\Lambda$ and injectivity radius bounded below by $\tau$. 
 Let $M^3 \looparrowright N^4$ be a complete, two-sided stable capillary minimal proper immersion whose contact angle $\theta$ satisfies \eqref{eq:main-angle},
 then there exists a
constant $C=C(\theta)$ such that 
\[
 \sup_{p\in M}|h|(p)\min\{1,d_M(p,\partial M\setminus \partial N),\tau,(\sqrt{\Lambda})^{-1}\}\leq C.
\]
\end{corollary}



\medskip
\noindent\textbf
 {Strategy and proof outline.}
Inspired by \cite{ChodoshLi,CLMS,Mazet}, we investigate whether the
$\mu$-bubble method can be extended to  stable capillary minimal hypersurfaces
in $\R^4_+$. Two principal difficulties arise. First, the stability inequality (see Section 2)
contains the additional boundary term
\[
 -\cot\theta\int_{\partial M}h(\mu,\mu)f^2,
\]
which has no fixed sign. Consequently, the argument of Cao--Shen--Zhu
\cite{CaoShenZhu} does not directly imply one-endedness. Second, the Gulliver--Lawson conformal metric $r^{-2}g$ (see \cite{GL}) used in \cite{ChodoshLi} is incompatible with the boundary terms that appear in both the stability inequality under the conformal metric and the stability inequality for the $\mu$-bubble functional.

To overcome these difficulties, we introduce two new ingredients.  The first is
a symmetric tensor $P$ associated with a finite-energy Neumann harmonic function.
The quantity $P(\mu,\mu)$ on $\partial M$ agrees with the boundary integrand in
the Bochner--stability inequality, whereas its divergence is controlled by
the nonnegative interior terms in that inequality.  The
divergence theorem then shows that every such function is constant and hence
gives one-endedness in the required angle range; see
Section~\ref{sec:one-end}, especially Proposition~\ref{prop:L2-vanishing} and Corollary~\ref{cor:one-end}.

The second ingredient is a new radial-angular conformal metric $\rho^{-2}g$, $\rho=rV(\frac{x_4}{r})^{-2\tau}$, replacing the radial
conformal metric $r^{-2}g$. Together with the Robin boundary condition for the positive
Jacobi function and the Gauss--Bonnet theorem, this conformal construction
provides the boundary and spectral control needed for the relative warped
$\mu$-bubble argument in the required angle range.

With the one-endedness result and the conformal construction in hand, we obtain that the scalar curvature has a strictly positive spectral lower bound in $(M,\widehat g)$.  At the
boundary, the transformed Robin coefficient satisfies the exact cancellation
\[
\widehat H_{\partial M}-\widehat\beta
=
-\frac32\langle\nabla\rho,\mu\rangle\ge0
\qquad\text{on }\partial M.
\]
Then, we construct a free boundary warped $\mu$-bubble in a compact domain of $(M,\widehat g)$. By
combining the stability inequality of the $\mu$-bubble functional with the Robin spectral diameter estimate,
we obtain uniform bounds for its area and diameter.  Converting back to $g$ and applying the Sobolev inequality, we conclude that $M$ has intrinsic
Euclidean volume growth.  Finally, the integral curvature estimate
proved in Appendix~B yields $h\equiv0$, completing the proof of Theorem~\ref{thm:main}. In addition, we make a comparison with the curvature
estimate of Li--Zhou--Zhu \cite[Appendix~C]{LiZhouZhu}, see Remark \ref{compare-LiZhouZhu}.

\medskip
\noindent\textbf{Disclosure on AI assistance.}
The authors used AI-assisted tools, principally ChatGPT 5.6, to help
identify the function yielding the exact cancellation of the boundary terms,
and verify some calculations. The authors verified and
completed all mathematical arguments and take full responsibility for the
content.

\section{Preliminaries on capillary minimal hypersurfaces}
\label{sec:preliminaries}
In this section, we fix our conventions and recall some preliminaries. By
passing to the universal cover, we may assume throughout that $M$ is simply
connected, since the hypotheses lift and flatness descends.
Let
\[
 \R^4_+:=\{x=(x_1,x_2,x_3,x_4)\in\R^4:x_4\geq0\},
 \qquad
 \partial\R^4_+=\{x_4=0\}.
\]
Let $M^3$ be a smooth manifold with boundary, and let
\[
 X:(M,\partial M)\longrightarrow(\R^4_+,\partial\R^4_+)
\]
be a two-sided proper immersion. For convenience, we identify $M$
with its image and do not distinguish between them. 
Let $\nu$ be
a global unit normal vector field of $M$, let
$E_4=(0,0,0,1)$, and let $\mu$ be the
outward unit co-normal vector field of $\partial M$ in $M$. Choose the unit normal $\bar\nu$ of the immersed surface
$X|_{\partial M}:\partial M\to\partial\R^4_+$ so that the ordered bases $\{\nu,\mu\}$ and
$\{\bar\nu,-E_4\}$ have the same orientation in the normal bundle of $\partial M$.
With the same contact angle convention as in \cite[Section~2]{WangXia}, we have
\begin{equation*}
 \mu=-\sin\theta\,E_4+\cos\theta\,\bar\nu,
 \qquad
 \nu=\cos\theta\,E_4+\sin\theta\,\bar\nu.
\end{equation*}
Thus $\langle\nu,E_4\rangle=\cos\theta$ and
$E_4^T=-\sin\theta\,\mu$.  Equivalently, $\theta$ is the angle between
$\nu$ and $E_4$, or between $\mu$ and $\bar\nu$.
Throughout, $C_c^\infty(M)$ denotes the space of smooth functions whose
supports are compact subsets of $M$ and may meet $\partial M$.
Denote the second
fundamental form and the mean curvature of $X:M\to \mathbb{R}^4_+$  respectively, by

\[
 h(U,V)=\langle\bar\nabla_U\nu,V\rangle,  \qquad H=\tr_g h.
\]
Denote the second
fundamental form and the mean curvature of $\partial M\subset M$, respectively, by
\[
 h^{\partial M}(U,V)=\langle\nabla_U \mu,V\rangle, \qquad  H^{\partial M}=\tr_{g^{\partial M}} h^{\partial M}.
\]
The contact angle condition together with the flatness of $\partial\mathbb R^4_+$ yields
\begin{equation}
 h^{\partial M}=\cot\theta\,h|_{T(\partial M)}.
 \label{eq:capillary-boundary}
\end{equation}
In particular, $\mu$ is a principal direction of $h$ along $\partial M$, as in
\cite[Proposition~2.1]{WangXia}.  Minimality
and the Gauss equation also give
\begin{equation*}
 R_M=-|h|^2,
 \qquad \Ric_M(U,U)=-|h(U,\cdot)|^2.
\end{equation*}
Let $X_{t}:(-\epsilon,\epsilon)\times M\to \R^{n+1}_{+}$ be a family of admissible compactly supported variation, whose variational vector field is $Y$ at $t=0$. Let $f=\bangle{Y,\nu}$, the first variation formula (see \cite{Souam1997}) of $E(t)$ is
\begin{align*}
    E^{\prime}(0)=\int_{M}HfdA+\int_{\partial M}\langle{Y,\mu-\cos{\theta}\bar{\nu}}\rangle ds.
\end{align*}
Thus the critical points are precisely the minimal hypersurfaces
that meet the supporting hyperplane at a constant contact angle
$\theta\in(0,\pi)$. For a capillary minimal hypersurface in $\R^{n+1}_{+}$, the second variation formula of $E(t)$ is
\begin{align*}
    E^{\prime\prime}(0)=-\int_{M}\left(\Delta f+|h|^2f\right)fdA+\int_{\partial M}\left(\frac{\partial f}{\partial\mu}-qf\right)fds,
\end{align*}
where $f\in C_c^{\infty}(M)$ and
$q=\cot\theta\,h(\mu,\mu)$.
From minimality and \eqref{eq:capillary-boundary}, we see that
\begin{equation*}
 q=-H^{\partial M}.
\end{equation*}
Let $K\subset M$ be a smooth connected compact subset with boundary $\partial K$. We call
$\partial K\cap\mathring M$ the \emph{interior boundary} of $K$.
The portion $\partial K\cap\partial M$ lies in $\partial M$ and is not counted
in the relative perimeter. We always assume $\partial K\cap\mathring M$ meets
$\partial M$
orthogonally. We write $\mathring M=M\setminus\partial M$; for a relative
domain $K\subset M$, $\mathring K$ denotes its interior relative to $M$.
The closures, connected components, boundaries, and
perimeters are understood relative to the capillary hypersurface $M$.

At the end of this section, we establish Sobolev inequalities for stable capillary immersed hypersurfaces in $\R^{n+1}_{+}$.
These inequalities hold for all $\theta\in(0,\pi)$.

\begin{lemma}
Let $n\geq3$ and let $M^{n}$
be a complete two-sided capillary minimal immersion in $\R^{n+1}_{+}$, with contact angle
$\theta\in(0,\pi)$. Writing $E_{n+1}=(0,\ldots,0,1)$, one has
\begin{equation}
-\sin\theta\int_{\partial M}\zeta
=\int_M\langle\nabla\zeta,E_{n+1}^T\rangle
\label{eq:trace-identity}
\end{equation}
for every $\zeta\in C_c^\infty(M)$. Moreover,
\begin{equation}
\left(\int_M|\zeta|^{n/(n-1)}\right)^{(n-1)/n}
\leq C_{n,\theta}\int_M|\nabla\zeta|,
\label{eq:L1-Sobolev}
\end{equation}
and, for every $\psi\in C_c^\infty(M)$,
\begin{equation}
\left(\int_M|\psi|^{2n/(n-2)}\right)^{(n-2)/n}
\leq C_{n,\theta}\int_M|\nabla\psi|^2.
\label{eq:L2-Sobolev}
\end{equation}
In particular, $\Vol(M)=\infty$.
\end{lemma}

\begin{proof}
Choose a compact smooth relative domain $K$ such that
$\partial K\cap\mathring M$ is disjoint from $\supp\zeta$. Since $M$ is
minimal,
\[
\Div_M(E_{n+1}^T)=0 \quad\text{in } M,
\qquad
\langle E_{n+1}^T,\mu\rangle=-\sin\theta
\quad\text{on }\partial M.
\]
The divergence theorem applied to $\zeta E_{n+1}^T$ gives
\eqref{eq:trace-identity}. Applying the same identity to smooth
approximations of $|\zeta|$ yields
\[
\sin\theta\int_{\partial M}|\zeta|
\leq\int_M|\nabla\zeta|.
\]
The boundary Michael--Simon inequality applied to $X|_K$ gives
\[
\left(\int_M|\zeta|^{n/(n-1)}\right)^{(n-1)/n}
\leq C_n\left(
\int_M|\nabla\zeta|
+\int_M|H|\,|\zeta|
+\int_{\partial M}|\zeta|
\right);
\]
see \cite[Lemma~2.3]{WangZhang} and \cite{MichaelSimon}. Since $H=0$,
the preceding trace estimate proves \eqref{eq:L1-Sobolev}.

Applying \eqref{eq:L1-Sobolev} to
$|\psi|^{2(n-1)/(n-2)}$ and using the Cauchy--Schwarz inequality implies 
\eqref{eq:L2-Sobolev}.

Finally, suppose that $\Vol(M)<\infty$. Let $\chi_R$ be a standard cutoff
function satisfying
\[
0\leq\chi_R\leq1,
\qquad \chi_R\longrightarrow1,
\qquad |\nabla\chi_R|\leq\frac{C}{R}.
\]
Applying \eqref{eq:L2-Sobolev} to $\chi_R$ and letting $R\to\infty$, we obtain 
\[
\Vol(M)^{(n-2)/n}\leq0,
\]
a contradiction. Hence $\Vol(M)=\infty$.
\end{proof}

\section{Harmonic functions and one-endedness}
\label{sec:one-end}

In this section, assume that $n\geq3$ and let
\[
 X:(M^n,\partial M)\longrightarrow(\R^{n+1}_+,\partial\R^{n+1}_+)
\]
be a complete, connected, two-sided stable capillary minimal immersion. We
first prove that every finite-energy Neumann harmonic function is constant
under a suitable restriction on the contact angle, and then deduce that $M$
has exactly one end.

\begin{proposition}
\label{prop:L2-vanishing}
If the contact angle $\theta$ satisfies
\begin{equation}
 \frac{(n-2)|\cos\theta|}{\sin^2\theta}<1,
 \label{eq:one-end-angle}
\end{equation}
then every smooth function $f$ on $M$ satisfying
\[
 \Delta f=0  \quad\text{in } M,\qquad
 \nabla_\mu f=0\quad\text{on }\partial M,
 \qquad \int_M|\nabla f|^2<\infty,
\]
is constant.
\end{proposition}

\begin{proof}
Throughout the proof, \(\nabla|\nabla f|\) denotes the weak gradient, represented by zero on the critical set. 

On the regular set $\{|\nabla f|>0\}$, set
\[
 T:=\frac{\nabla f}{|\nabla f|}.
\]
Since $f$ is harmonic, the Gauss equation and the Bochner formula give
\begin{equation}
 \frac12\Delta|\nabla f|^2
 =|\nabla^2f|^2-|h(T,\cdot)|^2|\nabla f|^2.
 \label{eq:bochner}
\end{equation}
The Neumann boundary condition implies that $T\in T(\partial M)$ and
\[
 \partial_\mu|\nabla f|
 =-h^{\partial M}(T,T)|\nabla f|
 \quad\text{on }\partial M.
\]
Let
\[
 E:=\operatorname{span}\{T,\mu\}^{\perp}\subset T(\partial M).
\]
 Minimality and
\eqref{eq:capillary-boundary} give
\[
 |\nabla f|\partial_\mu|\nabla f|-q|\nabla f|^2
 =\cot\theta\,\tr_Eh\,|\nabla f|^2.
\]
The following products extend smoothly across the critical set through the expressions on the right.
Multiplying \eqref{eq:bochner} by $\phi^2$, integrating by parts, and
using stability with the test function $\phi|\nabla f|$, we obtain
\begin{equation}
 \begin{split}
 &\int_M\phi^2\left(
 |\nabla^2f|^2-|\nabla|\nabla f||^2
 +\bigl(|h|^2-|h(T,\cdot)|^2\bigr)|\nabla f|^2\right)\\
 &\qquad\leq
 \int_M|\nabla f|^2|\nabla\phi|^2
 +\cot\theta\int_{\partial M}
 \phi^2\tr_Eh\,|\nabla f|^2.
 \end{split}
 \label{eq:I-less-J}
\end{equation}
Next we handle the boundary term. We introduce the following symmetric $2$-tensor $P$ on $M$:
\begin{equation}
 P
 =-|\nabla f|^2h-h(\nabla f, \nabla f)g
 +df\otimes h(\nabla f, \cdot)+h(\nabla f, \cdot)\otimes df.
 \label{eq:general-stress}
\end{equation}
One sees from minimality and $\nabla_\mu f=0$ that, along $\partial M$, 
\begin{equation}
 \begin{aligned}
 P_{\mu\mu}
 &=-(h_{TT}+h_{\mu\mu})|\nabla f|^2
 =\tr_Eh\,|\nabla f|^2.
 \end{aligned}
 \label{eq:general-stress-boundary}
\end{equation}
Denote $Y=E_{n+1}^T$. Since $\mu=-\frac{1}{\sin \theta}Y$ along $\partial M$, 
the divergence theorem and
\eqref{eq:general-stress-boundary} give
\begin{align}
-\sin\theta
 \int_{\partial M}\phi^2\tr_Eh\,|\nabla f|^2
 &=\int_{\partial M}P(\mu, Y)\phi^2=\int_M \Div(\phi^2P(Y,\cdot))\nonumber\\&
 =\int_M2\phi\phi_iP_{ij}Y_j
 +\phi^2(\nabla_iP_{ij})Y_j-
 \phi^2\langle E_{n+1},\nu\rangle P_{ij}h_{ij}.\label{eq-boundary-interior}
\end{align}
 In a local orthonormal frame
$\{e_i\}_{i=1}^n$, write
\[
 f_i=\nabla_i f,\qquad
 f_{ij}=\nabla_i\nabla_jf,\qquad
 h_{ij}=h(e_i,e_j).
\]
Fix $p\in\{|\nabla f|>0\}$ and choose the geodesic orthonormal frame
so that $e_1=T$ and $e_n=\mu$ at $p$, so that
\[
 E=\operatorname{span}\{e_2,\ldots,e_{n-1}\}.
\]
Then at $p$,
\[
 f_1=|\nabla f|,
 \qquad f_2=\cdots=f_n=0.
\]
Since both $h$ and $\nabla^2f$ are trace-free,
\[
 h_{11}=-\sum_{\alpha=2}^nh_{\alpha\alpha},
 \qquad
 f_{11}=-\sum_{\alpha=2}^nf_{\alpha\alpha}.
\]
Since
\[
 \partial_j|\nabla f|
 =\frac{f_if_{ij}}{|\nabla f|}
 =f_{1j},
\]
we have
\begin{align*}
 |h|^2-|h(T,\cdot)|^2
 &=\sum_{\alpha,\beta=2}^nh_{\alpha\beta}^2
   +\sum_{\alpha=2}^nh_{1\alpha}^2,
\end{align*}
and
\begin{align*}
 |\nabla^2f|^2-|\nabla|\nabla f||^2
 &=\sum_{\alpha,\beta=2}^nf_{\alpha\beta}^2
   +\sum_{\alpha=2}^nf_{1\alpha}^2.
\end{align*}
Equation \eqref{eq:general-stress} gives
\[
 P_{11}=0,\qquad
 P_{1\alpha}=0,\qquad
 P_{\alpha\beta}
 =\left(\sum_{\gamma=2}^nh_{\gamma\gamma}
  \delta_{\alpha\beta}-h_{\alpha\beta}\right)|\nabla f|^2.
\]
A direct calculation and the Cauchy--Schwarz inequality yield
\begin{align}\label{xeq1}
 \frac{1}{|\nabla f|^4}\sum_{i,j=1}^nP_{ij}^2&=\sum_{\alpha,\beta=2}^n
 \left(\sum_{\gamma=2}^nh_{\gamma\gamma}
  \delta_{\alpha\beta}-h_{\alpha\beta}\right)^2=(n-3)
\left(\sum_{\gamma=2}^nh_{\gamma\gamma}\right)^2
 +\sum_{\alpha,\beta=2}^nh_{\alpha\beta}^2
 \\&\leq(n-2)^2\sum_{\alpha,\beta=2}^nh_{\alpha\beta}^2\le (n-2)^2 \left(|h|^2-|h(T,\cdot)|^2\right)\nonumber
\end{align}
and
\begin{align}\label{xeq2}
  \frac{1}{|\nabla f|^2}\left|\sum_{i,j=1}^nP_{ij}h_{ij}\right|
 &=\left|\left(\sum_{\alpha=2}^nh_{\alpha\alpha}\right)^2
   -\sum_{\alpha,\beta=2}^nh_{\alpha\beta}^2 \right |
   \\&\le(n-2)\sum_{\alpha,\beta=2}^nh_{\alpha\beta}^2 \le(n-2) \left(|h|^2-|h(T,\cdot)|^2\right)\nonumber.
\end{align}
 Using
\[
 \nabla_i h_{jk}=\nabla_j h_{ik},
 \qquad
 \nabla_i h_{ij}=0,
 \qquad
 f_{ii}=0,
\]
and differentiating \eqref{eq:general-stress}, we obtain
\begin{equation}
\begin{aligned}
 \nabla_iP_{ij}=h_{ik}f_{ik}f_j
 -h_{ji}f_{ik}f_k
 -f_{ji}h_{ik}f_k.
\end{aligned}
\label{eq:stress-divergence}
\end{equation}
Using \eqref{eq:stress-divergence}, for $j=1$ we obtain
\begin{align*}
 \frac{1}{|\nabla f|^2}\left|\sum_{i=1}^n\nabla_iP_{i1}\right|^2
 &=\left|\sum_{\alpha,\beta=2}^n
 \left[
 \left(\sum_{\gamma=2}^nh_{\gamma\gamma}\right)
 \delta_{\alpha\beta}-h_{\alpha\beta}
 \right]f_{\alpha\beta}\right|^2\\&\le (n-2)^2
 \left(\sum_{\alpha,\beta=2}^nh_{\alpha\beta}^2\right)
 \left(\sum_{\alpha,\beta=2}^nf_{\alpha\beta}^2\right).
\end{align*}
For $2\leq\alpha\leq n$, the same identity gives
\begin{align*}
  \frac{1}{|\nabla f|^2}\left|\sum_{i=1}^n\nabla_iP_{i\alpha}\right|^2
 &=\left|\left(\sum_{\beta=2}^nf_{\beta\beta}\right)h_{1\alpha}
   -\sum_{\beta=2}^nf_{\alpha\beta}h_{1\beta}+\left(\sum_{\beta=2}^nh_{\beta\beta}\right)f_{1\alpha}
   -\sum_{\beta=2}^nh_{\alpha\beta}f_{1\beta}\right|^2
   \\&\le 2(n-2)^2
 \left[
 \left(\sum_{\alpha,\beta=2}^nf_{\alpha\beta}^2\right)
 \left(\sum_{\alpha=2}^nh_{1\alpha}^2\right)
 +\left(\sum_{\alpha,\beta=2}^nh_{\alpha\beta}^2\right)
 \left(\sum_{\alpha=2}^nf_{1\alpha}^2\right)
 \right].
\end{align*}
Here we also used 
\begin{align*}
 &\sum_{\alpha,\beta=2}^n
 \left[
 \left(\sum_{\gamma=2}^nf_{\gamma\gamma}\right)
 \delta_{\alpha\beta}-f_{\alpha\beta}
 \right]^2
 \leq(n-2)^2\sum_{\alpha,\beta=2}^nf_{\alpha\beta}^2.
\end{align*}
It follows that
\begin{align}\label{xeq3}
 \frac{1}{|\nabla f|^2}\sum_{j=1}^n
 \left(\sum_{i=1}^n\nabla_iP_{ij}\right)^2
 &\leq(n-2)^2
 \left(\sum_{\alpha,\beta=2}^nh_{\alpha\beta}^2\right)
 \left(\sum_{\alpha,\beta=2}^nf_{\alpha\beta}^2\right)\nonumber\\
 &\quad+2(n-2)^2
 \left[
 \left(\sum_{\alpha,\beta=2}^nf_{\alpha\beta}^2\right)
 \left(\sum_{\alpha=2}^nh_{1\alpha}^2\right)
 +\left(\sum_{\alpha,\beta=2}^nh_{\alpha\beta}^2\right)
 \left(\sum_{\alpha=2}^nf_{1\alpha}^2\right)
 \right]\nonumber\\
 &\leq2(n-2)^2
 \left(
 \sum_{\alpha,\beta=2}^nh_{\alpha\beta}^2
 +\sum_{\alpha=2}^nh_{1\alpha}^2
 \right)
 \left(
 \sum_{\alpha,\beta=2}^nf_{\alpha\beta}^2
 +\sum_{\alpha=2}^nf_{1\alpha}^2
 \right)
 \\&=2(n-2)^2
 \bigl(|h|^2-|h(T,\cdot)|^2\bigr)
 \bigl(|\nabla^2f|^2-|\nabla|\nabla f||^2\bigr)\nonumber.
\end{align}
The preceding calculations  are restricted to \(\{|\nabla f|>0\}\). After multiplying by the corresponding powers of \(|\nabla f|\), the estimates also hold at critical points. So, the following estimate holds everywhere on \(M\).

By using \eqref{xeq1}, \eqref{xeq2} and \eqref{xeq3}, and noting that $|Y|\le 1$, we obtain
\begin{align}\label{xeq5}
&\left|2\phi\phi_iP_{ij}Y_j
 +\phi^2(\nabla_iP_{ij})Y_j-
 \phi^2\langle E_{n+1},\nu\rangle P_{ij}h_{ij}\right|
 \\&\le 2|\phi||\nabla \phi| \left(\sum_{i,j=1}^nP_{ij}^2\right)^{1/2}
 +\phi^2 \left[
 \sum_{j=1}^n
 \left(\sum_{i=1}^n\nabla_iP_{ij}\right)^2
 +\left(\sum_{i,j=1}^nP_{ij}h_{ij}\right)^2
 \right]^{1/2}\nonumber
 \\&\le 
 2(n-2)|\phi||\nabla \phi|\sqrt{|h|^2-|h(T,\cdot)|^2}\,
 |\nabla f|^2 \nonumber\\&\quad+ (n-2)\phi^2\left(
 |\nabla^2f|^2-|\nabla|\nabla f||^2
+\bigl(|h|^2-|h(T,\cdot)|^2\bigr)|\nabla f|^2
 \right)\nonumber
 \\&\le  (n-2+\epsilon)\phi^2\left(
 |\nabla^2f|^2-|\nabla|\nabla f||^2
+\bigl(|h|^2-|h(T,\cdot)|^2\bigr)|\nabla f|^2
 \right)+ C_{n,\epsilon}|\nabla f|^2|\nabla\phi|^2.\nonumber
 \end{align}
It follows from \eqref{eq-boundary-interior} and \eqref{xeq5} that
\begin{align}\label{xeq6}
&\left|\cot\theta
 \int_{\partial M}\phi^2\tr_Eh\,|\nabla f|^2\right|
 \\&= \frac{|\cos\theta}{\sin^2\theta|}\left|\int_M2\phi\phi_iP_{ij}Y_j
 +\phi^2(\nabla_iP_{ij})Y_j-
 \phi^2\langle E_{n+1},\nu\rangle P_{ij}h_{ij}\right|\nonumber
 \\&\le \frac{(n-2+\epsilon)|\cos\theta|}{\sin^2\theta}\int_M\phi^2\left(
 |\nabla^2f|^2-|\nabla|\nabla f||^2
 +\bigl(|h|^2-|h(T,\cdot)|^2\bigr)|\nabla f|^2\right)+ C_{n,\epsilon}\int_M|\nabla f|^2|\nabla\phi|^2.\nonumber
 \end{align}

Using \eqref{xeq6} in \eqref{eq:I-less-J}, by assuming $\theta$ satisfies \eqref{eq:one-end-angle}, we get
\begin{align}\label{xeq7}
 \int_M\phi^2\left(
 |\nabla^2f|^2-|\nabla|\nabla f||^2
 +\bigl(|h|^2-|h(T,\cdot)|^2\bigr)|\nabla f|^2
 \right)
 \leq
 C_{n,\theta}
 \int_M|\nabla f|^2|\nabla\phi|^2.
 \end{align}
Choose a standard cutoff function $\phi_R$ associated with the complete
metric and satisfying
\[
 \phi_R=1\ \text{on }B_R,\qquad
 \supp\phi_R\subset B_{2R},\qquad
 |\nabla\phi_R|\leq\frac{C}{R}.
\]
Applying \eqref{xeq7} with $\phi=\phi_R$ and using the refined
Kato inequality, we obtain
\[
\int_{B_R}|\nabla^2f|^2
\leq C_{n,\theta}\int_M|\nabla f|^2|\nabla\phi_R|^2
\leq \frac{C}{R^2}\int_M|\nabla f|^2
\longrightarrow0.
\]
Hence $\nabla^2f=0$ on $M$, and $|\nabla f|$ is constant. If
$|\nabla f|>0$, then
\[
 \int_M|\nabla f|^2
 =|\nabla f|^2\Vol(M)<\infty
\]
would imply $\Vol(M)<\infty$, contradicting
\eqref{eq:L2-Sobolev}. Therefore $|\nabla f|=0$, and $f$ is
constant.
\end{proof}

\begin{corollary}
\label{cor:one-end}
Let $M$ be a complete two-sided, stable capillary minimal hypersurface immersed in $\R^{n+1}_{+}$, $n\geq 3$. If the contact angle $\theta$ satisfies
\eqref{eq:one-end-angle}, then $M$ has exactly one end.
\end{corollary}

\begin{proof}
We adapt the one-end argument of Cao--Shen--Zhu
\cite{CaoShenZhu}, based on the Li--Tam theory
\cite{LiTam}, to the Neumann boundary condition.  Set
\[
 \mathcal D^{1,2}(M)
 :=\overline{C_c^\infty(M)}^{\,\|d(\cdot)\|_2}.
\]
By inequality \eqref{eq:L2-Sobolev},  \(\mathcal{D}^{1,2}(M)\) is a Hilbert space continuously embedded in \(L^{2n/(n-2)}(M)\).

Every end of $M$ has infinite volume. We argue by contradiction and suppose that some end has finite volume. Let $r$ be the distance to a smooth compact cross-section and set
\[
 V(t)=|\{r>t\}|,\qquad A(t)=P_M(\{r>t\}),
\]
where $P_M$ denotes the relative perimeter and does not count the portion
lying in $\partial M$. The coarea formula gives $-V'(t)=A(t)$ for almost every $t$.
Using strict approximation in $BV$, apply \eqref{eq:L1-Sobolev} to
$\chi_{\{t<r<T_j\}}$ and choose $T_j\to\infty$ such that $A(T_j)\to0$.
It follows that
\[
 V(t)^{(n-1)/n}\leq C_{n,\theta}A(t)
 =-C_{n,\theta}V'(t)
\]
for almost every $t$. Hence $(V^{1/n})'\leq-c<0$ as long as $V>0$. This would
force $V$ to vanish at a finite distance, contradicting the completeness and
unboundedness of the end.

If $M$ has at least two ends, choose a smooth function $\eta$ with compactly supported differential such that $\eta\equiv 0$ on one end and $\eta\equiv 1$ on another. By the Lax--Milgram theorem, there exists a unique
$u\in\mathcal D^{1,2}(M)$ satisfying
\[
 \int_M\langle du,d\psi\rangle
 =-\int_M\langle d\eta,d\psi\rangle,
 \qquad \psi\in\mathcal D^{1,2}(M).
\]
Then $f=\eta+u$ is a smooth harmonic function with finite Dirichlet energy satisfying the Neumann boundary condition
\[
\frac{\partial f}{\partial\mu}=0 \qquad\text{on }\partial M.
\]
The function $f$ is nonconstant. Indeed, if $f\equiv c$, then $u=c-\eta$.
On the two selected infinite-volume ends, $u$ equals $c$ and $c-1$,
respectively. Since $u\in L^{2n/(n-2)}(M)$, this would imply simultaneously
that $c=0$ and $c-1=0$, a contradiction. Thus $f$ is a nonconstant
finite-energy Neumann harmonic function, contradicting
Proposition~\ref{prop:L2-vanishing}.
\end{proof}
\begin{remark}
Condition~\eqref{eq:one-end-angle} is equivalent to
\begin{equation*}
 |\cos\theta|
 <\frac{\sqrt{(n-2)^2+4}-(n-2)}2.
\end{equation*}
Thus the preceding one-endedness result gives the following ranges.
\begin{center}
\small
\begingroup
\renewcommand{\arraystretch}{1.2}
\begin{tabular}{c|c|c}
ambient half-space & equivalent cosine bound & angle interval\\ \hline
$\R^4_+$ & $|\cos\theta|<(\sqrt5-1)/2$
 & $51.827292^\circ<\theta<128.172708^\circ$\\
$\R^5_+$ & $|\cos\theta|<\sqrt2-1$
 & $65.530199^\circ<\theta<114.469801^\circ$\\
$\R^6_+$ & $|\cos\theta|<(\sqrt{13}-3)/2$
 & $72.375609^\circ<\theta<107.624391^\circ$
\end{tabular}
\endgroup
\end{center}
In particular, for $n=3$, under condition \eqref{eq:main-angle}, $M^3\subset\mathbb{R}^4_+$ has one end, since condition \eqref{eq:main-angle} implies condition \eqref{eq:one-end-angle} when $n=3$.
\end{remark}
\section{Conformal deformation and warped \texorpdfstring{$\mu$}{mu}-bubbles}
\label{sec:conformal}
In this section, we construct an angular conformal metric and 
obtain uniform area and diameter bounds for relative warped $\mu$-bubbles with respect to this new metric.
\subsection{Conformal Deformation }
We first choose a fixed point
$a\in\partial\R^4_+\setminus X(\partial M)$,
set
\begin{equation}
 r=|X-a|,\qquad x_4=\langle X,E_4\rangle,\qquad
 z=\frac{x_4}{r}\in[0,1].
 \label{eq:angular-coordinates}
\end{equation}
In particular, $r>0$ everywhere on $M$.  Fix any $x_0\in M$.  Since the hypotheses and conclusion of Theorem~\ref{thm:main} are invariant under dilations centered at $a$, we rescale by the factor $r(x_0)^{-1}$ and hence assume
\[
r(x_0)=1.
\]

To control the boundary term under the conformal metric, we introduce an
angular factor depending on $z=x_4/r$.  Let $V=V(z)$ be the unique solution of the
ordinary differential equation
\begin{equation}
 (1-z^2)V''-2zV'-\frac14V=0
 \label{eq:angular-ODE}
\end{equation}
with initial data
\begin{equation}
 V(0)=\frac{\Gamma(1/4)^2}{2\pi^{3/2}},
 \qquad
 V'(0)=-\frac{\Gamma(3/4)^2}{\pi^{3/2}}.
 \label{eq:V-initial-data}
\end{equation}
Define
\begin{equation}
 p(z)=-2(\log V)'(z).
 \label{eq:p-definition}
\end{equation}
As proved in Appendix~\ref{app:angular-ode}, $V$ extends smoothly to
$[0,1]$ and, for $0\leq z\leq1$,
\begin{equation}
 V(1)=1,\qquad p(0)=\kappa_*,\qquad p>0,\qquad
 p'\leq-p^2,\qquad zp\leq\frac12.
 \label{eq:angular-properties}
\end{equation}
Moreover, \eqref{eq:angular-ODE} and \eqref{eq:p-definition} give
\begin{equation}
 (1-z^2)\bigl(2p'-p^2\bigr)=4zp-1.
 \label{eq:p-Riccati}
\end{equation}
In particular, $p>0$ implies that $V$ is decreasing, and hence
\begin{equation}
 1=V(1)\leq V(z)\leq V(0).
 \label{eq:V-bounds}
\end{equation}

Using the condition \eqref{eq:main-angle}, we choose $\tau$ such that
\begin{equation}
 \frac{|\cot\theta|}{\kappa_*}<\tau<1
 \label{eq:tau-choice}
\end{equation}
and define
\begin{equation}
 \rho=rV(z)^{-2\tau}.
 \label{eq:rho-definition}
\end{equation}
By \eqref{eq:V-bounds},
\begin{equation}
 V(0)^{-2\tau}r\leq\rho\leq r.
 \label{eq:rho-comparison}
\end{equation}
Next, let $\widehat{g}=\rho^{-2}g$ and rewrite the stability inequality under
the conformal metric $(M,\widehat{g})$.  Throughout this section, a hat denotes
a quantity computed with respect to $\widehat{g}$.
\begin{proposition}
    Let $M^{3}$ be a complete two-sided, stable capillary minimal hypersurface in $\R^{4}_{+}$ with the contact angle $\theta$ satisfying (\ref{eq:main-angle}). Then the following inequality holds in $(M,\widehat{g})$:
\begin{equation*}
\begin{aligned}
&\int_M
\left(
|\widehat\nabla\psi|_{\widehat g}^{2}
+\frac12\widehat R\,\psi^2
\right)dV_{\widehat g}
-\int_{\partial M}
\rho\left(
q+\frac12\langle\nabla\log\rho,\mu\rangle
\right)\psi^2\,dA_{\widehat g}\\
&\qquad\geq
\int_M
\rho^2\left(
\frac32\Delta\log\rho
-\frac34|\nabla\log\rho|^2
\right)\psi^2\,dV_{\widehat g}.
\end{aligned}
\end{equation*}
\end{proposition}
\begin{proof}
Let \(g\), $h$, $\mu$ and $R_M$ denote the induced metric, the second fundamental form, the outward unit conormal of \(\partial M\), and the scalar curvature of $M$, respectively. First,
the stability inequality implies the half-scalar inequality
\begin{equation*}
\int_M
\left(
|\nabla\phi|^2+\frac12 R_M\phi^2
\right)\,dV_g
-
\int_{\partial M} q\phi^2\,dA_g
\ge 0
\end{equation*}
for every \(\phi\in C_c^1(M)\).
Let \(\rho=rV(z)^{-2\tau}\) be the conformal factor constructed as above.
    Let $\widehat g=\rho^{-2}g$,
the conformal transformation formula of scalar curvature is
\begin{align*}
\widehat R_{M}
&=
\rho^2
\left(
R_M
+
4\Delta\log\rho
-
2|\nabla\log\rho|^2
\right).
\end{align*}
Let $\psi=\rho^{1/2}\phi$. 
Then,
\begin{align}\label{part-1}
\int_M
\left(
|\widehat\nabla\psi|_{\widehat g}^2
+
\frac12\widehat R_{M}\psi^2
\right)dV_{\widehat g}
={}&
\int_M
\left(
|\nabla\phi|^2+\frac12R_M\phi^2
\right)dV_g
+
\int_M
\phi\langle\nabla\phi,\nabla\log\rho\rangle dV_g\notag\\
&+
2\int_M
(\Delta\log\rho)\phi^2dV_g
-\frac34
\int_M
|\nabla\log\rho|^2\phi^2dV_g\notag\\
={}&
\int_M
\left(
|\nabla\phi|^2+\frac12R_M\phi^2
\right)dV_g\notag\\
&\quad+
\int_M
\left(
\frac32\Delta\log\rho
-
\frac34|\nabla\log\rho|^2
\right)
\phi^2dV_g\notag\\
&\quad+
\frac12
\int_{\partial M}
\phi^2\langle\nabla\log\rho,\mu\rangle\,dA_g.
\end{align}
On the other hand, we have
\begin{align}\label{part-2}
\int_{\partial M}
\rho
\left(
-q-\frac12\langle\nabla\log\rho,\mu\rangle
\right)
\psi^2dA_{\widehat g}
=
-\int_{\partial M}q\phi^2dA_g
-
\frac12
\int_{\partial M}
\phi^2\langle\nabla\log\rho,\mu\rangle\,dA_g.
\end{align}
The conclusion follows by adding (\ref{part-1}) to (\ref{part-2}) and applying the stability inequality.
\end{proof}
The next lemma is devoted to handling the terms arising from the stability inequality, which play a key role in the subsequent computations.
\begin{lemma}
\label{lem:rho-estimates}
The function $\rho$ satisfies
\[
 \rho^2\left(\frac32\Delta\log\rho
 -\frac34|\nabla\log\rho|^2\right)
 \geq\lambda,\quad\lambda:=\frac34(1-\tau)V(0)^{-4\tau}>0 \hbox{ in }M,
\]
and
\[ \langle\nabla\rho,\mu\rangle<0\hbox{ on }\partial M.\]
\end{lemma}
\begin{proof}
Since
$
\log \rho=\log r-2\tau\log V(z)$ and
$d\log\rho=d\log r+\tau p\,dz$,
we have
\[
\Delta\log\rho
=
\Delta\log r+\tau p\,\Delta z+\tau p'|\nabla z|^2
\]
and
\[
|\nabla\log\rho|^2
=
|\nabla\log r|^2
+2\tau p\langle\nabla\log r,\nabla z\rangle
+\tau^2p^2|\nabla z|^2.
\]
By minimality and a direct computation for $r$, $z$,
we have
\begin{align*}
r^2\Delta\log r
&=1+2\nu(r)^2,
&
r^2|\nabla\log r|^2
&=1-\nu(r)^2,
\\
r^2\Delta z
&=-z\bigl(2+\nu(r)^2\bigr)
  +2r\nu(r)\nu(z),
&
r^2|\nabla z|^2
&=1-z^2-r^2\nu(z)^2,
\\
r^2\langle\nabla\log r,\nabla z\rangle
&=-r\nu(r)\nu(z).
\end{align*}
Substituting these identities and using
$
(1-z^2)(2p'-p^2)=4zp-1,
$
we obtain
\[
\begin{aligned}
\rho^2\left(
\frac32\Delta\log\rho
-\frac34|\nabla\log\rho|^2
\right)
={}&
V^{-4\tau}
\Bigg\{
\frac34(1-\tau)
\left[1+\tau(1-z^2)p^2\right]\\
&\qquad
+\frac34\tau(\tau p^2-2p')r^2\nu(z)^2\\
&\qquad
+\left(\frac{15}{4}-\frac32\tau zp\right)\nu(r)^2
+\frac92\tau p\,r\nu(r)\nu(z)
\Bigg\}.
\end{aligned}
\]
By
$zp\le\frac12$,
$p'\le-p^2$,
$0<\tau<1$,
we have
$
\frac{15}{4}-\frac32\tau zp\ge3
$
and
$$
\frac34\tau(\tau p^2-2p')
\ge
\frac34\tau(\tau+2)p^2.
$$
Moreover,
\[
\begin{aligned}
\left(\frac{15}{4}-\frac32\tau zp\right)
\frac34\tau(\tau p^2-2p')
-
\left(\frac94\tau p\right)^2
\ge
\frac9{16}\tau(8-5\tau)p^2\ge0.
\end{aligned}
\]
Hence the last three terms are nonnegative. Therefore
\[
\begin{aligned}
\rho^2\left(
\frac32\Delta\log\rho-\frac34|\nabla\log\rho|^2
\right)
&\ge
\frac34(1-\tau)V(0)^{-4\tau}
=\lambda,
\end{aligned}
\]
where we used \(V(z)\le V(0)\).

It remains to prove the boundary inequality. Along \(\partial M\),
we have \(z=0\) and
\[
\mu=-\sin\theta\,E_4+\cos\theta\,\bar\nu.
\]
Since both \(X\) and \(a\) lie in the supporting hyperplane,
\[
|\partial_\mu r|\le|\cos\theta|,
\qquad
\partial_\mu z=-\frac{\sin\theta}{r}.
\]
Since \(p(0)=\kappa_*\), we obtain
\[
\frac{\partial_\mu\rho}{\rho}
=
\frac{\partial_\mu r-\tau\kappa_*\sin\theta}{r}
\le
\frac{|\cos\theta|-\tau\kappa_*\sin\theta}{r}.
\]
By
$
\frac{|\cot\theta|}{\kappa_*}<\tau,
$
the right-hand side is strictly negative. Hence
$
\langle\nabla\rho,\mu\rangle
=
\partial_\mu\rho<0.
$
\end{proof}

We now derive a Robin spectral inequality for the conformal metric.
\begin{equation*}
 \int_M\left(|\widehat\nabla\psi|^2
 +\frac12\widehat R_{M}\psi^2\right)dV_{\widehat g}
 +\int_{\partial M}\widehat\beta\psi^2dA_{\widehat g}
 \geq\lambda\int_M\psi^2dV_{\widehat g},
\end{equation*}
where
\begin{equation*}
 \widehat\beta=\rho\left(-q-\frac12\langle\nabla\log\rho,\mu\rangle\right).
\end{equation*}
Hence, we have $\lambda_{1}(-\widehat\Delta+\frac{1}{2}\widehat{R}_{M})\geq \lambda$, and the corresponding Jacobi function $u$ satisfies 
$$\partial_{\widehat\mu} u+\widehat{\beta }u=0,\quad \text{on}\quad \partial M.$$


Next, we use $u$ as the warped function to construct free boundary $\mu$-bubbles. Using the stability of the functional defined as below, we will estimate the diameter and area of the bubbles, assuming that the scalar curvature on $(M,\widehat{g})$ satisfies a spectral lower bound. Let us recall some preliminary results on free boundary $\mu$-bubbles.

\subsection{Free boundary warped \texorpdfstring{$\mu$}{mu}-bubbles}
Let $(K,\partial K)\subset(M,\partial M)$ be a smooth compact manifold with piecewise smooth boundary. We let $\partial K=\partial_{0}K\cup\partial_{+}K\cup \partial_{-}K$, where $\partial_{-}K$ and $\partial_{+}K$ are disjoint and the interior angles between $\partial_{+}K$, $\partial_{-}K$ and $\partial_{0}K$ are less than or equal to $\frac{\pi}{2}$. This angle condition ensures that the minimizer of functional $\mathcal{A}(\Omega)$ remains away from the corners. If necessary, Lemma 4.3 of \cite{Chen-Hong2025} allows us to perturb $K$ near the corners so that this angle condition is satisfied.

Let $\tilde{h}$ be a smooth function on $K\setminus(\partial_{-}K\cup\partial_{+}K)$, with $\tilde{h}\to\pm\infty$ on $\partial_{\pm}K$. We choose a regular value $c_{0}$ of $\tilde{h}$ on $K\setminus(\partial_{-}K\cup\partial_{+}K)$ and take $\Omega_{0}=\tilde{h}^{-1}((c_{0},\infty))$ as the reference set. 

We choose the function $u$ as above and consider the following functional: 
\begin{equation*}
 \mathcal A_{u,\widetilde h}(\Omega)
 =
 \int_{\partial^*\Omega\cap\mathring K}u
 -
 \int_{K}
 (\chi_\Omega-\chi_{\Omega_0})\widetilde h\,u.
\end{equation*}
where the Caccioppoli sets $\Omega\subset K$ satisfy $\Omega\bigtriangleup\Omega_{0}\Subset K\setminus(\partial_{+}K\cup\partial_{-}K)$.
We minimize among relative Caccioppoli sets agreeing with $\Omega_0$ near the
two level boundaries. Since $\widetilde h$ tends to opposite infinities at
these boundaries and $\partial_{\widehat\mu}\widetilde h=0$ on $\partial M$, the
existence and barrier results in \cite[Proposition~15]{ChodoshLiSoap} and
\cite[Lemma~6.2]{Wu} give a stable minimizer $\Omega$ whose reduced boundary
lies in a compact subband and meets $\partial M$ orthogonally. By the regularity theory for free boundary hypersurfaces
\cite{ChodoshEdelenLi,DePhilippisMaggi,
DePhilippisMaggiDimensional,Wu,Grueter}, we know
$\partial^*\Omega$ is smooth. Here we only state the relevant conclusions without proof.
\begin{proposition}[{\cite[Lemma 6.2]{Wu}}]
\label{prop:free-boundary-bubble-existence}
The functional $\mathcal{A}_{u,\widehat{h}}$ admits a minimizer $\Omega$. If $n\leq 7$, the reduced boundary is smooth and meets $\partial_{0}K$ orthogonally.
\end{proposition}
All geometric quantities below are computed with respect to the conformal metric $(M,\widehat{g})$. Let $\widehat{A}_{\Sigma}$, $\widehat{H}_{\Sigma}$ be the second fundamental form and the mean curvature of $\Sigma$, respectively. We denote $\widehat{\nu}_{\Sigma}$ by the outward unit normal vector field of $\Sigma$ inside $\Omega$, denote $\widehat{\nu }_{\partial\Sigma}$ by the unit co-normal vector field of $\partial\Sigma$. Let $\widehat{I}_{\partial M}$ be the second fundamental form of $\partial M$ in $M$ with respect to the outward unit normal $\widehat{\mu }$. The first and the second variation formulas are as follows.
\begin{proposition}[{%
  \cite[Proposition~15]{ChodoshLiSoap} ,
  \cite[Theorem~6.3]{Wu}%
}]Let $\Omega$ be a smooth minimizer of $\mathcal{A}_{u,\tilde{h}}$. Assume that the corresponding normal variational vector field is $\psi\widehat\nu_{\Sigma}$ at $t=0$, then we have
\begin{align*}
\delta\mathcal{A}_{u,\tilde{h}}(\psi)=\int_{\Sigma}\langle{\widehat\nabla u,\widehat\nu_{\Sigma}}\rangle\psi+u\widehat{H}_{\Sigma}\psi-\tilde{h}u\psi+\int_{\partial\Sigma}u\langle\psi\widehat\nu_{\Sigma},\widehat\nu_{\partial\Sigma}\rangle.
\end{align*}
Hence, the critical point of the functional satisfies that
\[
\widehat{H }_{\Sigma}+\langle{\widehat{\nabla}\log u,\widehat\nu_{\Sigma}}\rangle-\tilde{h}=0\quad \text{in}\quad\Sigma,\quad \widehat\nu_{\Sigma}\in T_{x}\partial_{0}K\quad \text{for}\quad x\in\partial\Sigma\subset\partial_{0}K.
\]
The second variation formula is 
\begin{align}\label{eq:raw-relative-second-variation}
 \delta^2\mathcal A_{u,\widetilde h}(\psi)
 ={}&
 \int_\Sigma u|\widehat\nabla\psi|^2
 -\frac12
 (\widehat{R}_M-2\widehat{K}_\Sigma+|\widehat{A}_\Sigma|^2+\widehat{H }_\Sigma^2)u\psi^2\notag\\
 &+\int_\Sigma
 (\widehat{\Delta}_Mu-\widehat{\Delta }_\Sigma u-\widehat{\nu }_\Sigma(\widetilde h\,u))\psi^2-\int_{\partial\Sigma}
 \widehat{I}_{\partial M}(\widehat{\nu}_\Sigma,\widehat{\nu }_\Sigma)u\psi^2.
\end{align}
\end{proposition}

We now proceed to construct warped $\mu$-bubbles on compact domains corresponding to the above functional, from which we derive diameter and area estimates. 
\begin{lemma}
\label{lem:relative-bubble}
Let $K\subset M$ be a compact, connected, smooth relative domain whose interior
boundary
\[
 \Gamma=\partial K\cap\mathring M
\]
is nonempty, connected, and meets $\partial M$ orthogonally. If
$C\subset\mathring K$ is a nonempty, connected, compact set with
\begin{equation*}
 d_K(C,\Gamma)>\frac{5\pi}{\sqrt\lambda},
\end{equation*}
then there is a compact, connected relative collar
$\mathcal C\subset K\setminus C$, containing a one-sided neighborhood of
$\Gamma$ in $K$, such that
\begin{equation*}
 \partial\mathcal C\cap\mathring M
 =\Gamma\sqcup\mathcal B,
 \qquad
 \mathcal B=\Sigma_1\sqcup\cdots\sqcup\Sigma_m,
 \quad m\geq1.
\end{equation*}
The inner boundary $\mathcal B$ is contained in the
$5\pi/\sqrt\lambda$ neighborhood of $\Gamma$, and every $\Sigma_j$ is a
smooth, two-sided, stable warped $\mu$-bubble component meeting
$\partial M$ orthogonally. Each component $\Sigma=\Sigma_j$ is
diffeomorphic to either $\mathbb{S}^2$ or $\mathbb{D}^2$ and satisfies
\begin{equation}
 \begin{array}{c|c|c}
 \Sigma&|\Sigma|&\diam\Sigma\\ \hline
 \mathbb{S}^2&\leq8\pi/\lambda&\leq2\pi/\sqrt\lambda\\
 \mathbb{D}^2&\leq4\pi/\lambda&\leq2\pi/\sqrt\lambda.
 \end{array}
 \label{eq:bubble-bounds}
\end{equation}
More precisely,
\begin{equation}
 \mathcal C\subset
 \{x\in K:d_K(x,\Gamma)<5\pi/\sqrt\lambda\}.
 \label{eq:bubble-collar-location}
\end{equation}
\end{lemma}

\begin{proof}

Let $d(x)=d_K(x,\Gamma)$ be the distance to $\Gamma$. Following the forcing construction in
\cite[Lemma~6.1]{ChodoshLi}, with an even smoothing in boundary Fermi charts,
we may choose a smooth function $d_0$ such that $a$, $b$ and $(a+b)/2$ are regular values, and the function $d_0$ satisfies
\[
 |d_0-d|<\delta,\qquad
 |\widehat\nabla d_0|\leq2,\qquad
 \partial_{\widehat{\mu}} d_0=0\quad\text{on }\partial M,
\]
and
\[
 \sup_\Gamma d_0<a<b<\inf_C d_0,
 \qquad
 b-a>\frac{4\pi}{\sqrt\lambda},
 \qquad
 b+\delta<\frac{5\pi}{\sqrt\lambda}.
\]
In particular, the level hypersurfaces meet $\partial M$ orthogonally. On the
open band $\{a<d_0<b\}$, set
\[
 \varphi=\frac{\pi(d_0-a)}{b-a}-\frac\pi2,
 \qquad
 \widetilde h=-\sqrt\lambda\tan\varphi.
\]
Then $|\widehat{\nabla}\varphi|<\sqrt\lambda/2$ and
\begin{equation}
 \lambda+\widetilde h^2-2|\widehat{\nabla}\widetilde h|>0.
 \label{eq:tangent-forcing}
\end{equation}

Let $\Omega_0=\{\varphi<0\}$ and consider the relative warped $\mu$-bubble
functional
\begin{equation*}
 \mathcal A_{u,\widetilde h}(\Omega)
 =
 \int_{\partial^*\Omega\cap\{a<d_0<b\}\cap\mathring M}u
 -
 \int_{\{a<d_0<b\}}
 (\chi_\Omega-\chi_{\Omega_0})\widetilde h\,u.
\end{equation*}
By Proposition~\ref{prop:free-boundary-bubble-existence}, there is a  minimizer $\Omega$ for this functional.
We denote  $\mathcal C$ by the connected component of
\[
\Omega\cup\{x\in K:d_0(x)\leq a\}
\]
containing $\Gamma$. The barrier inequalities imply that
$\mathcal C\subset K\setminus C$ , and
\[
 \partial\mathcal C\cap\mathring M
 =\Gamma\sqcup\mathcal B,
 \qquad
 \mathcal B=\Sigma_1\sqcup\cdots\sqcup\Sigma_m,
 \quad m\geq1.
\]
Each $\Sigma_j$ is stable by local minimality, and
$d\leq d_0+\delta<b+\delta$ on $\mathcal C$ gives
\eqref{eq:bubble-collar-location}.

Since $|\widehat{A}_\Sigma|^2\geq \widehat{H}_\Sigma^2/2$ and
$\widehat{H}_\Sigma=\widetilde h-\langle\widehat{\nabla}\log u,\widehat{\nu}_\Sigma\rangle$, one has
\[
 -\frac34u\widehat{H}_\Sigma^2-\widehat{\nu}_\Sigma(\widetilde h\,u)
 =
 -u\widehat{\nu}_\Sigma(\widetilde h)
 -\frac12u\widetilde h^2
 -\frac12u\bangle{\widehat{\nabla}\log u,\widehat{\nu}_{\Sigma}}^2
 -\frac14u\widehat{H}_\Sigma^2.
\]
Therefore,  
the second variation formula~\eqref{eq:raw-relative-second-variation} gives
\begin{equation}
 \begin{aligned}
 0\leq{}&
 \int_\Sigma u|\widehat{\nabla}\psi|^2
 -\frac12(\widehat{R}_M-\lambda-2\widehat{K}_\Sigma)u\psi^2+(\widehat{\Delta}_Mu-\widehat{\Delta}_\Sigma u)\psi^2
 -\frac12 u^{-1}(\widehat{\nu}_\Sigma u)^2\psi^2\\
 &-\frac12\int_\Sigma
 (\lambda+\widetilde h^2+2\widehat{\nu}_\Sigma(\widetilde h))u\psi^2-\int_{\partial\Sigma}
 \widehat{I}_{\partial M}(\widehat{\nu}_\Sigma,\widehat{\nu}_\Sigma)u\psi^2.
 \end{aligned}
 \label{eq:relative-bubble-second-variation}
\end{equation}
Taking $\psi=u^{-1/2}$, using \eqref{eq:tangent-forcing} and
\[
 \frac{\widehat{\Delta}_Mu}{u}
 \leq\frac12\widehat{R}_M-\lambda,
\]
 we obtain
\begin{equation*}
 \begin{aligned}
 0\leq{}&
 \int_\Sigma\left(
 \widehat{K}_\Sigma-\frac\lambda2
 -\frac34|\widehat{\nabla}_\Sigma\log u|^2
 -\frac12\bangle{\widehat{\nabla}\log u,\widehat{\nu}_{\Sigma}}^2
 -\widehat{\Delta}_\Sigma\log u\right)-\int_{\partial\Sigma}
 \widehat{I}_{\partial M}(\widehat{\nu}_\Sigma,\widehat{\nu}_\Sigma).
 \end{aligned}
\end{equation*}
Integrating by parts and using 
$ -\partial_{\widehat\mu} \log u=\widehat{\beta}$,
we get
\begin{equation*}
 \begin{aligned}
 0\leq{}&
 \int_\Sigma\left(
 \widehat{K}_\Sigma-\frac\lambda2
 -\frac34|\widehat{\nabla}_\Sigma\log u|^2
 -\frac12\bangle{\widehat{\nabla}\log u,\widehat{\nu}_{\Sigma}}^2
\right)+\int_{\partial\Sigma}
 \left(\hat{\beta}-\widehat{I}_{\partial M}(\widehat{\nu}_\Sigma,\widehat{\nu}_\Sigma)\right).
 \end{aligned}
\end{equation*}
Note that
$$
\widehat{H}_{\partial M}
=
\widehat{\kappa}
+
\widehat{I}_{\partial M}(\widehat{\nu}_\Sigma,\widehat{\nu}_\Sigma),$$
holds along $\partial\Sigma$, where $\widehat{\kappa}$ denotes the geodesic curvature of $\partial\Sigma$ in
$\Sigma$ with respect to $\widehat{\mu}$. Consequently, we have
\[
\widehat{\beta}-\widehat{I}_{\partial M}(\widehat{\nu}_\Sigma,\widehat{\nu}_\Sigma)
=
\widehat{\kappa}-(\widehat{H}_{\partial M}-\widehat{\beta}).
\]
It follows that
\begin{equation}
 \begin{aligned}
 0\leq{}&
 2\pi\chi(\Sigma)+\int_\Sigma\left(
 -\frac\lambda2
 -\frac34|\widehat{\nabla}_\Sigma\log u|^2
 -\frac12\bangle{\widehat{\nabla}\log u,\widehat{\nu}_{\Sigma}}^2
\right)-\int_{\partial\Sigma}
(\widehat{H}_{\partial M}-\widehat{\beta}).\label{xeq8}
 \end{aligned}
\end{equation}
Here we also used the  Gauss--Bonnet formula. 

Recall that 
$H_{\partial M}=-q$. Hence
\begin{equation*}
 \widehat H_{\partial M}=\rho(-q-2\langle\nabla\log\rho,\mu\rangle).
\end{equation*}
Consequently,
\begin{align}
 \widehat H_{\partial M}-\widehat\beta
 &=\rho\left(-q-2\langle\nabla\log\rho,\mu\rangle\right)
   -\rho\left(-q-\frac12\langle\nabla\log\rho,\mu\rangle\right)
  =-\frac32\langle\nabla\rho,\mu\rangle\geq0,\label{xeq9}
\end{align}
where the last inequality is from Lemma \ref{lem:rho-estimates}.
It follows from \eqref{xeq8} and \eqref{xeq9} that
\begin{align*}
 \frac\lambda2|\Sigma|
 +\int_\Sigma\frac{3}{4}|\widehat{\nabla}_\Sigma\log u|^2
 +\frac12\langle\widehat{\nabla}\log u,\widehat{\nu}_\Sigma\rangle^2
 \leq2\pi\chi(\Sigma).
\end{align*}
The left-hand side is strictly positive, so $\chi(\Sigma)>0$. Since $\Sigma$
is connected and two-sided, it is diffeomorphic to either $\mathbb S^2$ or
$\mathbb D^2$. The corresponding area bounds follow immediately.

Finally, it follows from
\eqref{eq:relative-bubble-second-variation} that, for every
$w\in C^{\infty}(\Sigma)$,
\begin{align*}
 0\leq{}
 \int_\Sigma u|\widehat{\nabla} w|^2
 -\left(\frac\lambda2-\widehat{K}_\Sigma\right)uw^2
 -(\widehat{\Delta}_\Sigma u)w^2-\int_{\partial\Sigma}
 \widehat{I}_{\partial M}(\widehat{\nu}_\Sigma,\widehat{\nu}_\Sigma)uw^2.
 \end{align*}
Let $w>0$ be the first eigenfunction of this quadratic form. Since its first
eigenvalue is nonnegative, we obtain
\[
 \Div_\Sigma(u\widehat{\nabla} w)
 +\left(\frac\lambda2-\widehat{K}_\Sigma\right)uw
 +w\widehat{\Delta}_\Sigma u
 \leq0,
 \qquad
 \partial_{\widehat{\mu}} w
 =\widehat{I}_{\partial M}(\widehat{\nu}_\Sigma,\widehat{\nu}_\Sigma)w.
\]
For $f=uw$, the identity
\[
 2uw\langle\widehat\nabla u,\widehat\nabla w\rangle
 =
 |\widehat\nabla f|^2-w^2|\widehat\nabla u|^2-u^2|\widehat\nabla w|^2
 \leq|\widehat\nabla f|^2
\]
gives
\[
 \widehat{\Delta}_\Sigma f
 \leq
 -\left(\frac\lambda2-\widehat{K}_\Sigma\right)f
 +\frac{|\widehat{\nabla} f|^2}{2f},
 \qquad
 \partial_{\widehat\mu} f+\widehat{\kappa} f
 =(\widehat{H}_{\partial M}-\widehat{\beta})f\geq0.
\]
Apply the Robin spectral diameter estimate \cite[Lemma~8.2]{Wu}, with
$K_0=\lambda/2$, we have
\[
 \diam(\Sigma)\leq\frac{2\pi}{\sqrt\lambda}.
\]
This completes the proof.
\end{proof}

\section{Intrinsic Euclidean volume growth and Proof of Theorem \ref{thm:main}}
\label{sec:simons}
\begin{proposition}\label{volumegrowth}
    Let $X:(M^3,\partial M)\longrightarrow
 (\R^4_+,\partial\R^4_+)$ be a complete, connected, two-sided stable capillary minimal immersion
with contact angle $\theta$ satisfying \eqref{eq:main-angle}. Then $X$ has intrinsic Euclidean volume growth, i.e., there exists some $C=C_{\theta}$ so that \begin{equation*}
\Vol(B_M(R))\leq C R^3,\quad \forall R>0.
\end{equation*}
\end{proposition} 
\begin{proof}Let $x_0$ be the point fixed after
\eqref{eq:angular-coordinates}, and set
\[
 s(x)=d_M(x_0,x).
\]

Since $V(0)^{-2\tau}r\leq\rho\leq r$ and
$r(x)\leq r(x_0)+s(x)=1+s(x)$,
\begin{equation}
 \rho(x)\leq1+s(x).
 \label{eq:rho-linear}
\end{equation}
The function $\rho$ is homogeneous of degree one under ambient dilations
and has a smooth angular factor; hence
\begin{equation}
 |\widehat\nabla\log\rho|_{\widehat g}
 =|\nabla\rho|_g\leq C_\theta.
 \label{eq:rho-log-gradient}
\end{equation}
Moreover, \((M,\widehat g)\) is complete. Indeed, let \(\gamma:[0,\infty)\to M\) be a \(g\)-unit-speed smooth curve that eventually leaves every compact subset of \(M\). Since \(s(\gamma(t))\leq s(\gamma(0))+t\), inequality \eqref{eq:rho-linear} gives
\[L_{\widehat g}(\gamma)
=\int_0^\infty\frac{dt}{\rho(\gamma(t))}
\geq
\int_0^\infty
\frac{dt}{1+s(\gamma(t))}
\geq
\int_0^\infty
\frac{dt}{1+s(\gamma(0))+t}
=\infty.\]
Since \(M\) is complete, the sublevel sets of \(s\) are compact. Choosing a smooth  function $f$ such that
\[
 |f-s|<\frac12,\qquad
 \partial_\mu f=0\quad\text{on }\partial M.
\]
Fix a large constant \(L=L(\theta)\). For sufficiently large \(R\), choose a compact connected smooth domain \(\widetilde K_R\) whose interior boundary meets \(\partial M\) orthogonally and such that
\[
\overline{B_M(x_0,LR-2)}
\subset\mathring{\widetilde K}_R,
\qquad
\widetilde K_R\subset B_M(x_0,LR+2).
\]
Let \(K_R\) be the union of \(\widetilde K_R\) and all bounded components of \(M\setminus\widetilde K_R\). Since \(M\) has one end, then \(M\setminus K_R\) is connected. Moreover,
\[
B_M(x_0,LR-2)\subset\mathring K_R,
\qquad
\Gamma_R:=\partial K_R\cap\mathring M
\subset\{LR-2<s<LR+2\},
\]
and \(\Gamma_R\) meets \(\partial M\) orthogonally.

The hypersurface $\Gamma_R$ is connected.  Indeed, if it had two distinct
components, the connectedness of $K_R$ and $M\setminus\mathring K_R$ would yield
a closed curve in $M$ intersecting one component of $\Gamma_R$ transversely exactly once, contradicting the simple connectivity of \(M\).
Let $\gamma:[0,\ell]\to M$ be a curve parametrized by
$g$-arclength, joining $\overline{B_M(x_0,R)}$ to $\Gamma_R$.
By \eqref{eq:rho-linear} and $(s\circ\gamma)'\leq1$ almost
everywhere,
\begin{equation}
\begin{aligned}
L_{\widehat g}(\gamma)
&=\int_0^\ell\frac{dt}{\rho(\gamma(t))}
\geq\int_0^\ell
\frac{(s\circ\gamma)'(t)}{1+s(\gamma(t))}\,dt\\
&=\log\frac{1+s(\gamma(\ell))}{1+s(\gamma(0))}
\geq\log\frac{LR-1}{R+1}.
\end{aligned}
\label{eq:conformal-collar-distance}
\end{equation}
Choose $L$ such that $\log L>5\pi/\sqrt{\lambda}$.
For sufficiently large $R$, \eqref{eq:conformal-collar-distance}
gives
\[
d_{\widehat g}\bigl(\overline{B_M(x_0,R)},\Gamma_R\bigr)
>\frac{5\pi}{\sqrt{\lambda}}.
\]
 Apply Lemma~\ref{lem:relative-bubble} to $(M,\widehat g)$ with
\[
 K=K_R,\qquad \Gamma=\Gamma_R,
 \qquad C=\overline{B_M(x_0,R)}.
\]
This yields a compact connected collar $\mathcal C_R\subset K_R$ with
\begin{equation*}
 \partial\mathcal C_R\cap\mathring M
 =\Gamma_R\sqcup\mathcal B_R,
 \qquad
 \mathcal B_R=\Sigma_R^1\sqcup\cdots\sqcup\Sigma_R^{m_R}
\end{equation*}
where every $\Sigma_R^j$ is a $\mu$-bubble component and
\begin{equation*}
 \mathcal C_R\subset
 \{d_{K_R,\widehat g}(\cdot,\Gamma_R)
       <5\pi/\sqrt{\lambda}\}.
\end{equation*}
In particular, $\mathcal C_R$ is disjoint from
$\overline{B_M(x_0,R)}$.

Let $U_R$ be the component of $K_R\setminus\mathcal C_R$
containing $\overline{B_M(x_0,R)}$.
Since $\mathcal B_R$ is a properly embedded, two-sided free
boundary hypersurface, $\partial U_R\cap\mathring M$ is a union
of connected components of $\mathcal B_R$.

The set $(M\setminus\mathring K_R)\cup\mathcal C_R$ is connected
and meets the closure of every component of
$K_R\setminus\mathcal C_R$ other than $U_R$.
Hence $M\setminus U_R$ is connected.

Since $\overline{U_R}$ is compact and $M$ is noncompact,
the interior boundary
\[
\Sigma_R=\partial U_R\cap\mathring M
\]
is nonempty. The same argument used for $\Gamma_R$, together
with the simple connectivity of $M$, shows that $\Sigma_R$
is connected. Thus $\Sigma_R$ is a component of
$\mathcal B_R$, and \eqref{eq:bubble-bounds} gives
\[
|\Sigma_R|_{\widehat g}\leq\frac{8\pi}{\lambda}.
\]

On $\Gamma_R$, \eqref{eq:rho-linear} gives $\rho\leq C_\theta R$.
Every point of $\Sigma_R$ lies at uniformly bounded
$\widehat g$-distance from $\Gamma_R$. Hence
\eqref{eq:rho-log-gradient} gives $\rho\leq C_\theta R$
on $\Sigma_R$. Therefore
\begin{equation}
|\Sigma_R|_g
=\int_{\Sigma_R}\rho^2\,dA_{\widehat g}
\leq C_\theta R^2.
\label{eq:g-area}
\end{equation}

Applying \eqref{eq:L1-Sobolev} to smooth approximations of
$\chi_{U_R}$, we obtain
\begin{equation}
|U_R|^{2/3}
\leq C_\theta|\partial U_R\cap\mathring M|
=C_\theta|\Sigma_R|_g.
\label{eq:relative-isoperimetry}
\end{equation}
Since $B_M(x_0,R)\subset U_R$, equations \eqref{eq:g-area}
and \eqref{eq:relative-isoperimetry} imply
\begin{equation*}
\Vol(B_M(x_0,R))\leq C_\theta R^3.
\label{eq:intrinsic-growth}
\end{equation*}
\end{proof}
\noindent{\it Proof of Theorem \ref{thm:main}.} 
According to the contact condition \eqref{eq:main-angle}, we have
\[
 \frac{|\cos\theta|}{\sin^2\theta}
 <\frac{c_*}{1-c_*^2}
 =0.5023919707\ldots
 <\frac{9\sqrt3}{28},
\]
By Propositions \ref{volumegrowth} and \ref{prop:capillary-curvature-estimate}, $X$ is flat.
\qed

\appendix
\section{The angular ODE}
\label{app:angular-ode}

We prove here the properties of $p$  in
\eqref{eq:angular-properties}.  Put $t=(1-z)/2$.  Then
\eqref{eq:angular-ODE} becomes
\begin{equation*}
 t(1-t)V_{tt}+(1-2t)V_t-\frac14V=0.
\end{equation*}
 After normalizing
its value at $t=0$ to be $1$, its coefficients $a_k$ satisfy
\begin{equation*}
 a_0=1,
 \qquad
 (k+1)^2a_{k+1}=\left(k+\frac12\right)^2a_k.
\end{equation*}
The  evaluation of this series and its derivative at
$t=1/2$, together with $\frac{d}{dz}=-\frac12 \frac{d}{dt}$, gives the corresponding value at $z=0$,
\begin{equation*}
 V(0)=\frac{\Gamma(1/4)^2}{2\pi^{3/2}},
 \qquad
 V'(0)=-\frac{\Gamma(3/4)^2}{\pi^{3/2}}.
\end{equation*}
These are precisely the initial data in \eqref{eq:V-initial-data}.
Uniqueness for the initial value problem at $z=0$ therefore shows that the
solution defined in the main text extends smoothly to $z=1$ and satisfies
$V(1)=1$.

Equation \eqref{eq:angular-ODE} can be written as
\begin{equation}
 \bigl((1-z^2)V'(z)\bigr)'=\frac14V(z).
 \label{eq:V-divergence-form}
\end{equation}
Since $V$ is smooth at $z=1$, integration from $z$ to $1$ gives
\begin{equation}
 (1-z^2)V'(z)=-\frac14\int_z^1V(t)\,dt.
 \label{eq:V-integrated}
\end{equation}
Since $V(1)=1$, $V$ is positive near $z=1$.  If $V$ vanished in
$[0,1)$, let $z_0$ be its largest zero.  Then $V>0$ on $(z_0,1]$, whereas
\eqref{eq:V-integrated} gives $V'(z_0)<0$, a contradiction.  Hence $V>0$
on $[0,1]$.  Equation \eqref{eq:V-integrated} then gives $V'<0$ on $[0,1)$,
which also proves \eqref{eq:V-bounds} and the positivity assertion in
\eqref{eq:angular-properties}.

The initial value gives
\begin{equation*}
 p(0)=-2\frac{V'(0)}{V(0)}
 =4\frac{\Gamma(3/4)^2}{\Gamma(1/4)^2}
 =\kappa_*.
\end{equation*}
Moreover, setting $z=0$ in
\eqref{eq:V-integrated} yields
\begin{equation}
 p(0)=\frac{1}{2V(0)}\int_0^1V(t)\,dt<\frac12,
 \label{eq:p-origin-bound}
\end{equation}.

From \eqref{eq:p-definition}, equation \eqref{eq:angular-ODE} becomes
\eqref{eq:p-Riccati}, and therefore
\begin{equation}
 2(1-z^2)\bigl(p'+p^2\bigr)
 =3(1-z^2)p^2+4zp-1.
 \label{eq:p-comparison-identity}
\end{equation}
Evaluating \eqref{eq:angular-ODE} and its derivative at $z=1$ gives
\begin{equation*}
 V'(1)=-\frac18,
 \qquad
 V''(1)=\frac9{128},
\end{equation*}
and hence
\begin{equation*}
 p(1)=\frac14,
 \qquad
 p'(1)=-\frac7{64}.
\end{equation*}
Thus the right-hand side of \eqref{eq:p-comparison-identity} vanishes at
$z=1$, and its derivative there equals $3/16$.  It is therefore negative
for $z<1$ sufficiently close to $1$.

At any interior zero of the right hand side of
\eqref{eq:p-comparison-identity}, we have $p'=-p^2$ and $4zp<1$,
so its derivative is
\[
2p(1-zp)>0.
\]
Since it is
negative near $z=1$, it must be nonpositive throughout $[0,1]$,
proving the differential inequality in \eqref{eq:angular-properties}.
Finally, this differential inequality shows that $p$ is decreasing.
Together with the positivity of $p$ and \eqref{eq:p-origin-bound}, this gives
\begin{equation*}
 zp(z)\leq p(z)\leq p(0)<\frac12,
 \qquad 0\leq z\leq1,
\end{equation*}
and proves the last assertion in \eqref{eq:angular-properties}.

\section{An integral curvature estimate}
\label{app:curvature-estimate}
 Li-Zhou-Zhu \cite{LiZhouZhu} extended Schoen-Simon-Yau's result \cite{SchoenSimonYau} to capillary case. That is, they proved stable Bernstein theorem for capillary hypersurfaces in $\R^{n+1}$,  $3\leq n\leq 5$ provided the intrinsic Euclidean volume growth condition, when the contact angle $\theta$ belongs to some restricted interval. For example, for $n=3$, $\theta$ satisfies $$\frac{5}{3}+\frac{3\sqrt{2}}{2}+(\sqrt{2}-1)G_{\theta}-\frac{5\sqrt{2}}{2}G_{\theta}^{2}>0,\quad \hbox{ where }
G_{\theta}=\frac{(1+|\cos{\theta}|)^{3}}{\sin^{2}\theta}.
$$
The range above is $|\cos\theta|<0.0304949469\ldots.$
In the following, we extend the range of  $\theta$ for $n=3$.
\begin{proposition}\label{prop:capillary-curvature-estimate}
Let $X:(M^3,\partial M)\longrightarrow
 (\R^4_+,\partial\R^4_+)$ be a complete, connected, two-sided stable capillary minimal immersion
with contact angle $\theta$ satisfying
 $$\frac{|\cos\theta|}{\sin^2\theta}<\frac{9\sqrt3}{28},$$
Assume $M$ has intrinsic Euclidean volume growth, then $M$ is flat. 
\end{proposition}

\begin{proof}
Reverse the chosen normal if necessary and replace $\theta$ by
$\pi-\theta$, so that $\theta\leq\pi/2$. The Simons identity and the refined
Kato inequality for the trace-free Codazzi tensor $h$ imply on
$\{|h|>0\}$ that
\[
\frac12\Delta|h|^2=|\nabla h|^2-|h|^4,
\qquad
|\nabla h|^2\geq\frac53|\nabla|h||^2.
\]

 Since the co-normal vector $\mu$ is a principal
direction along $\partial M$, we choose an orthonormal frame $\{\mu,e_2,e_3\}$ that diagonalizes  $h$. By the
Codazzi equation and the capillary boundary condition, we have
\[
\nabla_\mu h_{ab}
=\cot\theta\bigl(h(\mu,\mu)h_{ab}-(h_T^2)_{ab}\bigr),
\qquad 2\leq a,b\leq 3.
\]
Taking the tangential trace and using $\operatorname{tr}h=0$, we obtain
\[
\nabla_\mu h(\mu,\mu)=\cot\theta\,|h|^2.
\]
Consequently,
\begin{equation}\label{boundary-identity-new}
\frac12\partial_{\mu}(|h|^2)
=\cot\theta\bigl(2|h|^2h(\mu,\mu)-\operatorname{tr}(h^3)\bigr).
\end{equation}

Fix $\varphi\in C_c^\infty(M)$ and write
\[
|h|_\varepsilon=(|h|^2+\varepsilon^2)^{1/2}.
\]
Direct differentiation gives
\[
|h|_\varepsilon\Delta|h|_\varepsilon+|\nabla|h|_\varepsilon|^2
=\frac12\Delta|h|^2=|\nabla h|^2-|h|^4,
\]
and hence
\begin{equation}\label{regularized-new}
|h|_\varepsilon\Delta|h|_\varepsilon+|h|^4
\geq\frac23|\nabla|h|_\varepsilon|^2
\end{equation}
everywhere on $M$.

Multiplying \eqref{regularized-new} by
$|h|_\varepsilon^{-2/3}\varphi^2$ and integrating by parts gives
\begin{align}
\int_M|h|_\varepsilon^{-2/3}|\nabla|h|_\varepsilon|^2\varphi^2
&+2\int_M|h|_\varepsilon^{1/3}\varphi
\langle\nabla|h|_\varepsilon,\nabla\varphi\rangle \notag\\
&\leq\int_M|h|^4|h|_\varepsilon^{-2/3}\varphi^2
+\int_{\partial M}|h|_\varepsilon^{1/3}
\partial_\mu|h|_\varepsilon\varphi^2.
\label{n3-first}
\end{align}
Testing the stability inequality with $|h|_\varepsilon^{2/3}\varphi$,
multiplying it by $3/2$, and using \eqref{n3-first}, we obtain
\begin{align}
\frac32\int_M|h|^2|h|_\varepsilon^{4/3}\varphi^2
-\int_M|h|^4|h|_\varepsilon^{-2/3}\varphi^2
&+\frac13\int_M|h|_\varepsilon^{-2/3}
|\nabla|h|_\varepsilon|^2\varphi^2 \notag\\
&\leq\frac32\int_M|h|_\varepsilon^{4/3}|\nabla\varphi|^2+\int_{\partial M}\varphi^2
\left(|h|_\varepsilon^{1/3}\partial_\mu|h|_\varepsilon
-\frac32q|h|_\varepsilon^{4/3}\right),
\label{n3-second}
\end{align}
where $q=\cot\theta\,h(\mu,\mu)$. On $\{|h|>0\}$,
\[
|h|_\varepsilon^{-2/3}|\nabla|h|_\varepsilon|^2
=|h|^2|h|_\varepsilon^{-8/3}|\nabla|h||^2,
\]
and these quantities increase as $\varepsilon\downarrow0$. Hence monotone
convergence applies to the gradient term. All interior zeroth-order terms
converge by dominated convergence. Moreover, \eqref{boundary-identity-new}
implies
\[
|h|_\varepsilon^{1/3}\partial_\mu|h|_\varepsilon
=\cot\theta\,|h|_\varepsilon^{-2/3}
\bigl(2|h|^2h(\mu,\mu)-\operatorname{tr}(h^3)\bigr).
\]
Its absolute value, together with
$|q||h|_\varepsilon^{4/3}$, is bounded on
$\partial M\cap\operatorname{supp}\varphi$ by
$C_\theta(|h|^2+1)^{7/6}$. Thus dominated convergence theorem also applies to the
boundary term. Letting $\varepsilon\downarrow0$ in \eqref{n3-second}, we obtain
\begin{align}
&\frac12\int_M|h|^{10/3}\varphi^2
+\frac13\int_{\{|h|>0\}}|h|^{-2/3}|\nabla|h||^2\varphi^2\notag\\
\leq&\frac32\int_M|h|^{4/3}|\nabla\varphi|^2 
+\cot\theta\int_{\partial M}|h|^{-2/3}
\left(\frac12h(\mu,\mu)|h|^2-\operatorname{tr}(h^3)\right)
\varphi^2.
\label{n3-limit}
\end{align}
Writing
$\kappa_1=h(\mu,\mu)$,
$\kappa_1+\kappa_2+\kappa_3=0$,
we have
\[
\frac12h(\mu,\mu)|h|^2-\operatorname{tr}(h^3)
=-(\kappa_2+\kappa_3)(\kappa_2-\kappa_3)^2
\]
and
\[
| (\kappa_2+\kappa_3)(\kappa_2-\kappa_3)^2 |
\leq\frac{4\sqrt2}{9}|h|^3.
\]
Apply (\ref{eq:trace-identity})
to $\zeta=|h|^{7/3}\varphi^2$. The Cauchy--Schwarz inequality gives
\begin{align*}
&\cot\theta\int_{\partial M}|h|^{-2/3}
\left(\frac12h(\mu,\mu)|h|^2-\operatorname{tr}(h^3)\right)\varphi^2\\
&\quad\leq\frac{4\sqrt2\cos\theta}{9\sin^2\theta}
\left[
\frac73
\left(\int_{\{|h|>0\}}|h|^{-2/3}|\nabla|h||^2\varphi^2\right)^{1/2}
\left(\int_M|h|^{10/3}\varphi^2\right)^{1/2}
\right.\\
&\quad\qquad\left.
+2\left(\int_M|h|^{4/3}|\nabla\varphi|^2\right)^{1/2}
\left(\int_M|h|^{10/3}\varphi^2\right)^{1/2}
\right].
\end{align*}
By Young's inequality,
\begin{align*}
&\frac{28\sqrt2\cos\theta}{27\sin^2\theta}
\left(\int_{\{|h|>0\}}|h|^{-2/3}|\nabla|h||^2\varphi^2\right)^{1/2}
\left(\int_M|h|^{10/3}\varphi^2\right)^{1/2}\\
&\qquad\leq
\frac13\int_{\{|h|>0\}}|h|^{-2/3}|\nabla|h||^2\varphi^2
+\frac{392\cos^2\theta}{243\sin^4\theta}
\int_M|h|^{10/3}\varphi^2.
\end{align*}
Combining this with \eqref{n3-limit} and the preceding boundary estimate gives
\begin{align}
\left(\frac12-\frac{392\cos^2\theta}{243\sin^4\theta}\right)
\int_M|h|^{10/3}\varphi^2
&\leq\frac32\int_M|h|^{4/3}|\nabla\varphi|^2 \notag\\
&\quad+
\frac{8\sqrt2\cos\theta}{9\sin^2\theta}
\left(\int_M|h|^{4/3}|\nabla\varphi|^2\right)^{1/2}
\left(\int_M|h|^{10/3}\varphi^2\right)^{1/2}.
\label{n3-after-first-young}
\end{align}
The coefficient on the left is positive precisely when
$
\frac{\cos^2\theta}{\sin^4\theta}<\frac{243}{784}$,
A second application of Young's inequality gives
\begin{align*}
&\frac{8\sqrt2\cos\theta}{9\sin^2\theta}
\left(\int_M|h|^{4/3}|\nabla\varphi|^2\right)^{1/2}
\left(\int_M|h|^{10/3}\varphi^2\right)^{1/2}\\
&\qquad\leq\frac12
\left(\frac12-\frac{392\cos^2\theta}{243\sin^4\theta}\right)
\int_M|h|^{10/3}\varphi^2
+\frac{64\cos^2\theta}
{81\sin^4\theta
\left(\frac12-\frac{392\cos^2\theta}{243\sin^4\theta}\right)}
\int_M|h|^{4/3}|\nabla\varphi|^2.
\end{align*}
Replacing $\varphi$ by $\varphi^{5/3}$ 
and applying H\"older's inequality, we obtain
\[
\begin{aligned}
\int_M|h|^{10/3}\varphi^{10/3}
\leq C_\theta
\int_M|h|^{4/3}\varphi^{4/3}|\nabla\varphi|^2\leq C_\theta
\left(\int_M|h|^{10/3}\varphi^{10/3}\right)^{2/5}
\left(\int_M|\nabla\varphi|^{10/3}\right)^{3/5}.
\end{aligned}
\]
Consequently,
\begin{align}\label{eq:curvature-Caccioppoli}
\int_M|h|^{10/3}\varphi^{10/3}
\leq C_\theta\int_M|\nabla\varphi|^{10/3}
\end{align}
 holds for every compactly supported smooth function \(\varphi\) on \(M\).
Let \(\chi_R\) be a smooth cutoff function such that
\[
 \chi_R=1\quad\text{on }B_M(x_0,R),\qquad
 \chi_R=0\quad\text{on }M\setminus B_M(x_0,2R),\qquad
 |\nabla\chi_R|\leq\frac2R.
\]
Using \(\phi=\chi_{R}\) in
\eqref{eq:curvature-Caccioppoli}, we obtain
\[
 \int_M|h|^{10/3}\chi_R^{10/3}
 \leq C_\theta\int_M|\nabla\chi_R|^{10/3}
 \leq C_\theta R^{-10/3}\Vol B_M(x_0,2R)
 \leq C_\theta R^{-1/3}.
\]
Letting $R\to \infty$, we get \(h\equiv0\), and \(M\) is flat.
\end{proof}

\begin{remark}\label{compare-LiZhouZhu}
We make a comparison for the technique used here with that of Li--Zhou--Zhu \cite[Appendix~C]{LiZhouZhu}. In \cite[Appendix~C]{LiZhouZhu}, they used the bound
\[
\left|\partial_\mu|h|^2\right|
\leq 6\sqrt{n-1}\,|\cot\theta|\,|h|^3
\]
and a trace inequality to estimate the boundary integral arising
from Simons' inequality. Combining this estimate with  stability inequality
gives their  curvature estimate.

For our derivation, we first combine the Simons and stability
inequalities so that the boundary terms partially cancel.
We then use
\[
\left|\frac12h(\mu,\mu)|h|^2-\operatorname{tr}(h^3)\right|
\leq\frac{4\sqrt2}{9}|h|^3.
\]
This gives a wider range of contact angles for $n=3$ than that
stated in \cite[Theorem~C.1]{LiZhouZhu}.
\end{remark}

\end{document}